\documentclass[10pt]{amsart}

\usepackage{amssymb}

\newtheorem{theorem}{Theorem}
\newtheorem{proposition}[theorem]{Proposition}
\newtheorem{lemma}[theorem]{Lemma}

\theoremstyle{definition}
\newtheorem*{problem*}{Problem}

\newtheorem*{maintheorem}{Main Theorem}
\newtheorem{question}{Question}
\theoremstyle{remark}

\newcommand{\C}{\mathbb C}
\newcommand{\R}{\mathbb R}

\newcommand{\M}{\mathcal M}
\newcommand{\g}{\mathfrak g}

\newcommand{\To}{\longrightarrow}
\newcommand{\isom}{\cong}

\begin{document}

\title[Zero divisors in algebras of Lie groups]{The problem of zero divisors in convolution algebras of supersolvable Lie groups}
\author{{\L}ukasz Garncarek}
\address{University of Wroc{\l}aw, Institute of Mathematics,
  pl.~Grunwaldzki 2/4\\ 50-384 Wroc{\l}aw\\ Poland}
\email{Lukasz.Garncarek@math.uni.wroc.pl}


\begin{abstract}
  We prove a variant of the Titchmarsh convolution theorem for simply
  connected supersolvable Lie groups, namely we show that the
  convolution algebras of compactly supported continuous functions and
  compactly supported finite measures on such groups do not contain
  zero divisors. This can be also viewed as a topological version of
  the zero divisor conjecture of Kaplansky.
\end{abstract}

\maketitle

\section{Introduction}
\label{sec:introduction}

The Titchmarsh convolution theorem asserts that if $f,g\in L^1(\R)$
vanish on $(-\infty,0)$, and $f*g(x) = 0$ for $x\leq T$, then there
exist real numbers $\alpha$ and $\beta$, such that $\alpha+\beta=T$,
and the functions $f$ and $g$ vanish on $[0,\alpha]$ and $[0,\beta]$
respectively. Existence of numerous proofs of this theorem in the
literature (\cite{Crum1941, Doss1988, Dufresnoy1947, Helson1983,
  Mikusinski1953, Mikusinski1959, Titchmarsh1926}) hints at its
significance. One of its corollaries states that the convolution
algebra $\M_c(\R)$ of compactly supported finite complex measures on
$\R$ has no zero divisors. It can be proved directly, by noticing that
the holomorphic Fourier transform defines an injective homomorphism of
$\M_c(\R)$ into $H(\C)$, the algebra of entire functions on $\C$ with
pointwise multiplication. Unlike the original theorem, which can not
be neatly adapted to a more general context of topological groups, the
formulation of this corollary makes sense for any such
group. In~\cite{Weiss_1968} it was proved for locally
compact abelian groups without nontrivial compact subgroups.

On the other hand, the Kaplansky zero divisor conjecture states that
the group algebra $K[G]$ of a discrete torsion-free group $G$ over a
field $K$ has no zero divisors. In the book~\cite{Passman1977} the
proofs of this conjecture in the cases of right-orderable,
supersolvable, and polycyclic-by-finite groups are presented.

A topological analogue of a torsion-free group is a compact-free
group, i.e. a~group without nontrivial compact subgroups. In view of
the aforementioned results, it seems reasonable to state the
following topological zero divisor problem:

\begin{problem*}
  Let $G$ be a compact-free topological group. Is it true that the
  convolution algebra $\M_c(G)$ contains no zero divisors?
\end{problem*}

Since the group algebra $\C[G]$ of $G$, viewed as a discrete group,
naturally embeds into $\M_c(G)$, an affirmative answer to this problem
implies that $G$ satisfies the original Kaplansky zero divisor
conjecture for fields of characteristic 0.

A general compact-free Lie group is of the form $R \rtimes \widetilde{SL_2}
(\R)^n$, where $R$ is simply connected and solvable, and $\widetilde{SL_2}(\R)$ is the universal cover of $SL_2(\R)$
(\cite{onishchik1994lie}, Theorem 3.2). We may thus break the original
problem into the following three questions:
\begin{question}\label{q:solvable}
  Let $G$ be a simply connected solvable Lie group. Can $\M_c(G)$
  contain zero divisors? 
\end{question}
\begin{question}
  Does $\M_c(\widetilde{SL_2}(\R))$ contain zero divisors?
\end{question}
\begin{question}
  Let $G_1$ and $G_2$ be Lie groups such that $\M_c(G_i)$ contain no
  zero divisors. Can $\M_c(G_1\rtimes G_2)$ contain zero divisors?
\end{question}
Negative answers to all three of them would yield a positive solution
to the topological zero divisor problem for all Lie groups. Actually,
since any linear subspace of a Lie algebra containing its commutant
is an ideal, a solvable Lie algebra decomposes as a direct sum of a
codimension $1$ ideal and a $1$-dimensional subalgebra. This gives a
decomposition of the corresponding Lie group into a semidirect
product, hence negative answers to questions 2 and 3 would be sufficient.

In this short paper we try to attack the first of the questions
above. We show the following:
\begin{maintheorem}
  Let $G$ be a connected, simply connected, supersolvable Lie group. Then
  $\M_c(G)$ has no zero divisors.
\end{maintheorem}
Our proof relies on some properties of holomorphic functions of one
variable. For solvable, but not supersolvable Lie groups, we
would have to consider holomorphic functions of several variables, in
which case the proof would break.

It is worth mentioning that the proof of the Kaplansky conjecture in
the poly\-cyclic-by-finite situation required completely different
methods than the supersolvable case (and supersolvable Lie groups are
somewhat analogous to supersolvable discrete groups). This may lead to
the supposition that also in the topological zero divisor problem a
different approach is required to deal with non-supersolvable Lie
groups.

We also remark that there is a variant of the Kaplansky conjecture in
characteristic $0$, known as the Linnell conjecture
(\cite{Linnell1991,Linnell1998}), which states that for a
discrete, torsion-free group $G$ and nonzero functions $f\in \C[G]$
and $g\in\ell^2(G)$ the convolution $f*g$ is nonzero. In
\cite{Ludwig2006} it was shown that the topological counterpart of
this conjecture, with $f\in C_c(G)$ and $g\in L^2(G)$, fails for every
nonabelian connected nilpotent Lie group.

A special case of the Main Theorem, with $G$ being the Heisenberg
group $H_n$, was proved and used in~\cite{GarncarekMSc} to show that
some natural representations of the group of contactomorphisms of an
arbitrary contact manifold are irreducible. We are currently working
on generalizing these results, so that the full power of the Main Theorem
could be utilized.

The paper is organized as follows. In
Section~\ref{sec:conv-algebr-locally} we discuss the convolution
algebras associated to a locally compact group. In
Section~\ref{sec:extensions-r} we show that if $\M_c(G)$ has no zero
divisors, then this property passes to all extensions of $G$ by
$\R$. Finally, in Section~\ref{sec:nilpotent-lie-groups} the solution
of the zero divisor problem for connected supersolvable Lie groups is
presented.

The author wishes to thank Jan Dymara, Jacques Faraut, Pawe{\l}
G{\l}owacki, Jean Ludwig, and the anonymous referee for their helpful
comments.

\section{Convolution algebras of locally compact groups}
\label{sec:conv-algebr-locally}

Let $G$ be a topological group. Denote by $\M_c(G)$ the set of all
compactly supported finite complex Borel measures on $G$. The convolution
$\mu * \nu$ of measures $\mu,\nu\in \M_c(G)$ is the measure defined by
the condition
\begin{equation}
  \int f(x)\,d(\mu*\nu)(x) = \iint f(xy)\,d\mu(x)d\nu(y).
\end{equation}
It is again compactly supported and finite, hence the operation of
convolution turns $\M_c(G)$ into an associative algebra. A continuous
group homomorphism $q\colon G \to H$ defines a pushforward map
$q_*\colon\M_c(G)\to \M_c(H)$, given by $q_*\mu(A) =
\mu(q^{-1}(A))$. It is a homomorphism of algebras.

Suppose that $K \leq G$ is a nontrivial compact subgroup. There exists
$1\ne a\in K$, and we have $(\delta_1-\delta_a) * \lambda_K = 0$,
where $\delta_x$ is the Dirac measure supported on $x$, and
$\lambda_K$ is the Haar measure on $K$. Since the inclusion of $K$
into $G$ is a continuous homomorphism, $\M_c(G)$ has zero
divisors. Hence the condition of being compact-free is necessary for
$\M_c(G)$ to contain no zero divisors.

From now on, suppose that $G$ is locally compact. We may thus fix a
left Haar measure $\lambda=\lambda_G$ on $G$. The modular function
$\Delta_G\colon G\to \R_+$ is then defined by $(R_x)_*\lambda =
\Delta_G(x)d\lambda$, where $R_x\colon G\to G$ is the right
translation by $x$. The space $C_c(G)$ of compactly supported
continuous complex-valued functions on $G$ with convolution
\begin{equation}
  f*g(x)=\int f(y)g(y^{-1}x)\,d\lambda(y)
\end{equation}
is also an algebra. It is embedded in $\M_c(G)$ through the
homomorphism $f \mapsto fd\lambda$. We will identify $C_c(G)$ with a
subalgebra of $\M_c(G)$ through this embedding.

\begin{lemma} \label{lem:CM_equiv}
  The following conditions are equivalent:
  \begin{enumerate}
  \item $\M_c(G)$ has no zero divisors, 
  \item $C_c(G)$ has no zero divisors.
  \end{enumerate}
\end{lemma}
\begin{proof}
  Suppose that $\mu,\nu\in\M_c(G)$ are nonzero, and $\mu*\nu = 0$. For
  any $f,g\in C_c(G)$ we then have
  \begin{equation}
    (f * \mu) * (\nu * g) = 0.
  \end{equation}
  But $f*\mu\in C_c(G)$, namely
  \begin{equation} \label{eq:f_mu}
    f*\mu(x) = \int f(xy^{-1})\Delta_G(y)\,d\mu(y).
  \end{equation}
  Notice, that the assignment $f(y) \mapsto f(y^{-1})\Delta_G(y)$ is a
  bijection of $C_c(G)$ with itself. In particular, if we set $x=e$ in
  equation~\eqref{eq:f_mu}, we may choose $f$ in such a way that
  $f*\mu(e)\ne 0$.

  Similarly, $\nu * g \in C_c(G)$, and we may choose $g$ so that
  $\nu*g\ne 0$. We therefore obtain a pair of zero divisors in
  $C_c(G)$. The other implication is obvious.
\end{proof}

If $N$ is a closed normal subgroup of $G$, and $q\colon G\to G/N$ is
the quotient map, then the measure pushforward map $q_*$ sends the
subalgebra $C_c(G)$ into $C_c(G/N)$. It is defined in this case by the
formula
\begin{equation}
  q_*f(xN) = \int_N f(xn)d\lambda_N(n).
\end{equation}

\section{Extensions by $\R$}
\label{sec:extensions-r}

Consider an extension $1\To A\To G\overset{q}\To Q\To 1$ of Lie
groups, such that $A\isom\R$. Choose a fixed identification of $A$
with $\R$. Denote by $\psi_+$ and $\psi_-$ the
indicator functions of intervals $[0,\infty)$ and $(-\infty,0]$ in
$A$, respectively. Now define operations $\gamma_+$ and $\gamma_-$ on
$C_c(G)$ by
\begin{equation}
  \label{eq:def-gamma}
  \gamma_\pm(f)(x)=\int_Af(xa^{-1})\psi_\pm(a)\,da,
\end{equation}
In general they do not preserve compact supports, however we have the following. 

\begin{lemma}\label{lem:int_cpt_supp}
  Suppose that $f\in C_c(G)$ satisfies\/ $q_*f=0$. Then the functions
  $\gamma_\pm(f)$ are also in $C_c(G)$. 
\end{lemma}
\begin{proof}
  Denote by $K$ the support of $f$. Then $\gamma_+(f)$ vanishes
  outside $KA$. Take $x\in G$ such that $\gamma_+(f)(x)\ne 0$ and
  write $x=ka$, where $k\in K$, and $a\in A$. By definition of
  $\gamma_+(f)$, there exists $b\in [0,\infty)$ such that
  $f(kab^{-1})\ne 0$, i.e. $kab^{-1}\in K$. We thus obtain $a-b\in
  K^{-1}K\cap A \subseteq [-D,D]$, where $D>0$ is a positive real
  number. Since $b>0$, we get $a \geq -D$, hence $\gamma_+(f)$ is
  supported in $K\cdot [-D,\infty)$. By a similar argument applied to
  $\gamma_-$, we infer that $\gamma_-(f)$ is supported in
  $K\cdot (-\infty,D]$. But
  \begin{equation}
    \gamma_+(f)(x)+\gamma_-(f)(x)=\int_Af(xa^{-1})\,da=q_*f(xA)=0,
  \end{equation}
  which implies that in fact both $\gamma_+(f)$ and $\gamma_-(f)$ have
  supports in $K\cdot [-D,D]$, which is compact.
\end{proof}

For $f\in C_c(G)$ and $x\in G$ define ${}_xf\colon A\to \C$ by
${}_xf(a)=f(xa)$. These functions are compactly supported on $A$, and
thus their Fourier transforms are holomorphic. We have
\begin{equation}
{}_x\gamma_+(f)(b)=\int_A{}_xf(ba^{-1})\psi_+(a)\,da=\int_{-\infty}^b{}_xf(t)\,dt,
\end{equation}
hence ${}_xf=({}_x\gamma_+(f))'$. If $q_*f=0$, then, by
Lemma~\ref{lem:int_cpt_supp} the Fourier transform of
${}_x\gamma_+(f)$ is defined and the following identity is satisfied:
\begin{equation}
  \label{eq:fourier-identity}
  ({}_xf)^\wedge(\chi) = i\chi ({}_x\gamma_+(f))^\wedge(\chi).
\end{equation}
Also, note that
\begin{equation}
  \label{eq:fourier-at-zero}
  ({}_xf)^\wedge(0)=\int_A f(xa)\,da=q_*f(xA).
\end{equation}





\begin{proposition}\label{prop:R_ext}
  Consider an extension $1\To A\To G\overset{q}\To Q\To 1$ of
  Lie groups such that $A\isom\R$. If $C_c(Q)$ has no zero
  divisors, then neither has $C_c(G)$.
\end{proposition}
\begin{proof}
  Assume to the contrary that there exist nonzero $f,g\in C_c(G)$ such
  that $f*g=0$. Suppose first, that $q_*g =
  0$. By~\eqref{eq:fourier-at-zero}, this means that every Fourier
  transform $({}_xg)^\wedge$ has a zero of order $n_x>0$ at $\chi=0$.
  At least one ${}_xg$ is nonzero, so at least one $n_x$ is
  finite. Let $n(g)=\min_xn_x$. By Lemma~\ref{lem:int_cpt_supp} and
  equation~\eqref{eq:fourier-identity}, each of the functions
  $\gamma_+^k(g)$, where $k=1,\ldots,n(g)$, is compactly supported,
  and $n(\gamma_+^k(g))=n(g)-k$. Therefore $q_*\gamma_+^{n(g)}(g)$ is
  nonzero. Furthermore, it is straightforward to see that
  $\gamma_+(f*g)=f*\gamma_+(g)$, hence $\tilde{g}=\gamma_+^{n(g)}(g)$
  is a new zero divisor, such that $f*\tilde{g}=0$ and $q_*\tilde{g}\ne0$.

  For an arbitrary locally compact group $H$ we may define an
  involution on $C_c(H)$ by $f^*=\Delta_H(x)\overline{f(x^{-1})}$. It satisfies
  $f^**g^*=(g*f)^*$, and commutes with homomorphisms induced by
  quotient maps. In particular, we have
  $\tilde{g}^**f^*=(f*\tilde{g})^*=0$, and
  $q_*(\tilde{g}^*)=(q_*\tilde{g})^*\ne 0$. We may proceed as before
  to replace $f^*$ with $\tilde{f}$ such that $q_*\tilde{f}\ne 0$, and
  $\tilde{g}^**\tilde{f}=0$. This leads to a contradiction, since
  $q_*$ is a homomorphism, so $q_*\tilde f$ and $q_*(\tilde g^*)$ are
  nontrivial zero divisors in $C_c(Q)$.
\end{proof}

\section{Supersolvable Lie groups}
\label{sec:nilpotent-lie-groups}

A real Lie algebra $\g$ is said to be supersolvable (also completely
solvable or triangular), if it contains a complete flag of
ideals, i.e.\ a chain $\g_0<\g_1<\cdots<\g_d=\g$ of ideals of $\g$ such
that $\dim\g_i=i$. Such an algebra is solvable and exponential (see
e.g.~\cite{onishchik1994lie}, Theorem 6.4). A Lie
group $G$ is supersolvable if its Lie algebra is supersolvable.

Supersolvability is a property which interpolates between solvability
and nilpotency. Obviously, any nilpotent Lie algebra is supersolvable,
as its lower central series can be refined to obtain a complete flag
of ideals. The simplest example of a Lie group which is
supersolvable, but not nilpotent, is the group ``$ax+b$'' of affine
transformations of the real line. Its Lie algebra $\g$ is spanned by two
vectors $X, Y$ such that $[X,Y]=Y$, and $0 \leq \R Y \leq \g$ is a
complete flag of ideals. 

Supersolvability is also strictly stronger than solvability. The
simplest example of a solvable Lie group which is not supersolvable is
the group $\mathop{\mathrm{Isom}}(\R^2)$ of affine isometries of the
Euclidean plane. Its Lie algebra is spanned by three vectors $X,Y,Z$
such that $[X,Y]=0$, $[Z,X]=Y$ and $[Z,Y]=-X$. It does not contain a
$1$-dimensional ideal, and therefore does not admit a complete flag of
ideals.

\begin{lemma}\label{lem:nilp_cfree}
  Let $G$ be a connected supersolvable Lie group. Then $G$ is compact-free
  if and only if it is simply connected.
\end{lemma}
\begin{proof}
  Let $G$ be simply connected. Its exponential
  map $\exp\colon\g\to G$ is then a~diffeomorphism. Suppose that
  $K\leq G$ is a nontrivial compact subgroup, and let $1\ne k\in
  K$. There exists $0\ne X\in\g$ such that $k=\exp X$. Since $K$ is
  compact, the sequence $\exp nX$ has an accumulation point, which is
  a contradiction, because the sequence $nX$ has no accumulation
  points in $\g$.

  Now, if $G$ is not simply connected, then it is a quotient of its
  universal cover $\widetilde G$ by a discrete normal subgroup $H$. If
  $1\ne h\in H$, then, since $\widetilde G$ is exponential, there
  exists a one-parameter subgroup of $\widetilde G$ containing $h$. It
  projects onto a compact one-parameter subgroup of $G$.
\end{proof}

\begin{lemma}\label{lem:simpl_conn_quot}
  The quotient of a connected, simply connected Lie group $G$ by a
  connected normal subgroup $N$ is simply connected.
\end{lemma}
\begin{proof}
  Since $G$ is simply connected, the quotient homomorphism $q\colon
  G\to G/N$ has a lift $\tilde q$ to the universal cover
  $\widetilde{G/N}$. It is surjective, because its image contains a~neighborhood of $1$, and $\widetilde{G/N}$ is
  connected. Furthermore, its kernel is a subgroup of $N$, and $\dim N
  = \dim\ker \tilde q$. Since $N$ is connected, we have $\ker\tilde q
  = N$, hence $G/N$ is isomorphic to its universal cover.
\end{proof}

The proof of the Main Theorem now becomes a mere formality:

\begin{proof}[Proof of the Main Theorem]
  By Lemma \ref{lem:CM_equiv} we may consider the convolution algebra
  $C_c(G)$ instead of $\M_c(G)$. We proceed by induction on $d=\dim
  G$. If $d=1$, then $G\isom\R$ and $C_c(\R)$ has no zero
  divisors. Now, suppose that $d>1$. Let $0=\g_0<\g_1<\cdots<\g_d=\g$
  be a~complete flag of ideals in the Lie algebra of $G$. The ideal
  $\g_1$ corresponds to a~closed normal subgroup $A\lhd G$, isomorphic
  to $\R$. The quotient $G/A$ is supersolvable and simply connected by
  Lemma~\ref{lem:simpl_conn_quot}, and $\dim G/A < d$. Hence
  $C_c(G/A)$ contains no zero divisors, and by
  Proposition~\ref{prop:R_ext}, $C_c(G)$ also has no zero divisors.
\end{proof}

\bibliographystyle{plain}
\bibliography{zerodivbib,/home/lukgar/library}

\end{document}